\documentclass[11pt]{article}
\usepackage[latin9]{inputenc}
\usepackage{amsmath}
\usepackage{amssymb}
\usepackage{esint}

\makeatletter
\@ifundefined{date}{}{\date{}}

\usepackage{mathrsfs}\usepackage{mathrsfs}\usepackage{amsthm}\usepackage{enumerate}\usepackage{array}\usepackage[titletoc]{appendix}

\numberwithin{equation}{section}
\allowdisplaybreaks[4]
\allowdisplaybreaks[4]
\newtheorem{theorem}{Theorem}[section]
\newtheorem{acknowledgement}[theorem]{Acknowledgment}

\newtheorem{definition}[theorem]{Definition}

\newtheorem{lemma}[theorem]{Lemma}

\newtheorem{proposition}[theorem]{Proposition}
\newtheorem{remark}[theorem]{Remark}

\newtheorem{hypothesis}{Hypothesis}[section]

\makeatother

\begin{document}

\title{Stochastic partial differential equations\\
 driven by space-time fractional noises}

\author{Ying Hu \\
IRMAR, Université Rennes 1\\
 35042 Rennes Cedex, France\\
 Email: Ying.Hu@univ-rennes1.fr\medskip{}
\\
 Yiming Jiang\\
 School of Mathematical Sciences, Nankai University\\
 Tianjin 300071, China\\
 Email: ymjiangnk@nankai.edu.cn \medskip{}
\\
 Zhongmin Qian\\
 Mathematical Institute, University of Oxford\\
 Oxford OX2 6GG, England\\
 Email: Zhongmin.Qian@maths.ox.ac.uk}

\maketitle

\begin{abstract}
In this paper, we study a class of stochastic partial differential
equations (SPDEs) driven by space-time fractional noises. Our method
consists in studying first the nonlocal SPDEs and showing then the
convergence of the family of these equations and the limit gives the
solution to the SPDE.
\end{abstract}



\section{Introduction}

Recently, stochastic partial differential equations are studied mainly
as alternative physical models for some complex and chaotic natural
phenomena. For example, in the study of turbulent phenomena, a reduction
must be made although the Navier-Stokes equations are believed to
catch the motions of all different sorts of flows of incompressible
fluids. It is thus hopeless, at least under the current technologies
and the current computational power (it might be changed with the
arrival of quantum computers which is still in the remote future),
to understand the solutions to the Navier-Stokes equations subject
to complicated boundary conditions with precision and mathematical
rigor. One of the ideas in the fluid dynamics is to combine the equations
of motions with statistical ideas. Statistical fluid mechanics has
been the main tool for the understanding of the turbulent flows. Traditional
statistical fluid mechanics is based on the hypothesis (which has
not been proved yet) that there is an underlying invariant measure
with respect to the non-linear semigroup defined by the Navier-Stokes
equations, and is not based on the use of stochastic evolution models.
It$\hat{o}$'s calculus and its generalizations to infinite dimensional
state spaces such as Malliavin Calculus etc., on the other hand, provide
the possibility to construct useful stochastic evolution models directly.
A class of simple models can be constructed by simplifying the equations
of motions and enhanced by adding suitable noise terms in order to
recover essential features in the original physical laws. For example,
for the equation of motion for an incompressible fluid:
\[
\frac{\partial u}{\partial t}+u\cdot\nabla u=\nu\Delta u-\nabla p\textrm{ and }\nabla\cdot p=0
\]
where $u$ describes the velocity of the flow and $p$ is the pressure.
In order to apply the familiar theory of parabolic equations, a simple
way to make reduction is to drop the pressure term $\nabla p$ from
the first equation, so that the Navier-Stokes equations become a parabolic
system
\[
\frac{\partial u}{\partial t}+u\cdot\nabla u=\nu\Delta u
\]
which preserves the non-linear convection, but certainly many interesting
features are lost. To recover the chaotic nature of the fluids, we
may add a noise term to the parabolic equation, which thus leads to
the following stochastic partial differential equations
\[
\frac{\partial u}{\partial t}+u\cdot\nabla u=\nu\Delta u+W
\]
where $W$ should be modeled by a space-time random field, and $W$
can be considered as a kind of radom perturbations, or as external
random force applied to the fluid in question. This kind of stochastic
partial differential equations has received study in the recent years,
and a lot of interesting results have been obtained. While, there
is no certain rules which dictate the choice of a noise term, and
the choice of a reasonable stochastic process really depends on the
equation of motion in question, by taking into account of the physical
meaning as far as possible.

In this paper, we study the following stochastic partial differential
equation (SPDE):

\begin{equation}
\left\{ \begin{array}{ll}
\frac{\partial u}{\partial t}(t,x)=\frac{1}{2}\Delta u(t,x)+g(u(t,x),\frac{\partial}{\partial x}u(t,x))+\dot{B}^{H}\text{,}\\
u(t,0)=0\text{,}\\
u(0,x)=u_{0}(x)\text{,}
\end{array}\right.\label{eq:intro}
\end{equation}
where $(t,x)\in\lbrack0,T]\times D$, $D=[0,\infty)$, and
\[
\dot{B}^{H}=B^{H}(dt,dx)=\frac{\partial^{2}B^{H}(t,x)}{\partial t\partial x}
\]
is a space-time fractional noise with Hurst parameter $H=(h_{1},h_{2})$,
$h_{i}\in(0,1)$ for $i=1,2$, see below for a precise definition.

SPDEs driven by Gaussian noises have been widely studied where the
non-linear term $g$ depends only on $u$, see for example Walsh \cite{Wa},
Da Prato-Zabczyk \cite{DZ}, Hu et al. \cite{HNS} and Fuhrman et
al. \cite{FHT}. These theories and their applications are now classic
and mature.

If $g$ has a particular form depending on $u$ and $\frac{\partial}{\partial x}u$
as well, such as the non-linear term in the Burgers equation, the
SPDE has been considered by many authors. Mohammed and Zhang \cite{Zhang12}
studied the dynamics of the stochastic Burgers equation on the unit
interval driven by affine linear noise. Using multiplicative ergodic
theory techniques, they established the existence of a discrete non-random
Lyapunov spectrum for the cocycle. They also proved an existence theorem
for solutions of the stochastic Burgers equation on the unit interval
subject to the Dirichlet boundary condition and the anticipating initial
velocities in \cite{Zhang13}. Wang et al. \cite{Wang} proposed an
$L^{2}$-gradient estimate for the corresponding Galerkin approximations,
and the $\log$-Harnack inequality was established for the semigroup
associated to a class of stochastic Burgers equations. Hairer and
Voss \cite{HV} discussed the numerical methods of various finite-difference
approximations to the stochastic Burgers equation.

Recently, a class of special Gaussian processes called fractional
Brownian motion (fBm) has been attracted attention due to their useful
feature of preserving long term memory, and a large number of interesting
results from scaling invariance to the description of their laws as
random fields have been established by various authors. The study
of these Gaussian processes has its historical motivation from their
applications in hydrology and telecommunication, and have been applied
to the mathematical finance, biotechnology and biophysics, see for
example \cite{OTHB,Gu,Kou} and the literature therein. Coutin and
Qian \cite{CQ}, Mandelbrot and Van Ness \cite{MV} and some other
authors have proposed a theory of stochastic calculus for a class
of continuous stochastic processes with long time memory, including
the fractional Brownian motions as arachtypical examples. Neuenkirch
and Tindel \cite{Neuenkirch} studied the least square-type estimator
for an unknown parameter in the drift coefficient of a stochastic
differential equation with additive fractional noise modeled by a
fractional Brownian motion with Hurst parameter $H>1/2$. Balan \cite{Balan}
identified necessary and sufficient conditions for the existence of
a random field solution for some linear stochastic partial differential
equations of parabolic and hyperbolic type. While, Bo et al. \cite{BJW2}
considered stochastic Cahn-Hilliard equations with fractional noises,
the existence, uniqueness and regularity of the solutions were obtained.
In \cite{JWZ} the stochastic Burgers equation driven by the fractional
noise was studied, a global mild solution was obtained and the existence
of a distribution density of the solution was also established.

The goal of this paper is to study SPDE (\ref{eq:intro}) where $g$
is a function depending on both $u$ and $\frac{\partial}{\partial x}u$,
and where
\[
\dot{B}^{H}=B^{H}(dt,dx)=\frac{\partial^{2}B^{H}(t,x)}{\partial t\partial x}
\]
is a space-time fractional noise. We study the existence and uniqueness
of solution to this class of SPDEs, and the regularity of its solution.

There are two key steps in the present approach. The first step is
to study the following non-local SPDE:
\begin{equation}
\left\{ \begin{array}{ll}
\frac{\partial u}{\partial t}(t,x)=\frac{1}{2}\Delta u(t,x)+g(u(t,x),u^{\theta}(t,x))+\dot{B}^{H}\text{,}\\
u^{\theta}(t,x)=\frac{1}{\theta}\left(u(t,x+\theta)-u(t,x)\right)\text{,}\\
u(t,0)=0\text{,}\\
u(0,x)=u_{0}(x)\text{,}
\end{array}\right.\label{eq:intro2}
\end{equation}
where $\theta\in\Bbb{R}$\textbf{, }$\theta\neq0$, is a parameter.
Note that $u^{\theta}$ is subject to the same boundary condition
as that of $u$. For each $\theta\neq0$, the unique solution to the
SPDE above depends on $\theta$, and thus is denoted by $u(t,x,\theta)$.
In the second step, we show that the family of solutions $u(t,x,\theta)$
converges in an appropriate function space to a limit $u(t,x)$ as
$\theta\rightarrow0$, which provides the solution of SPDE (\ref{eq:intro}).
A space-time fractional noise is a two-parameter Gaussian random field
which can be defined similarly as in one parameter case, and can be
specified in terms of its covariance function. A few definitions about
the stochastic integration theory for such space-time fractional noise
will be recalled in the following sections. As matter of fact, the
regularity properties of space-time fractional noises are fully reflected
in the Hurst parameter $H=(h_{1},h_{2})$, and our main result of
this paper can be simply described in terms of two parameters as following:
if $2h_{1}+h_{2}>2$ (which thus excludes the case of space-time white
noise for which $h_{1}=h_{2}=\frac{1}{2}$), then SPDE (\ref{eq:intro})
admits a unique solution which has nice regularity.

Hairer and Voss \cite{HV} studied a stochastic partial differential
equation where $g$ has a special form, driven by a space-time white
noise. While, in our setting we allow the non-linear term which depends
on both $u$ \ and its space derivative $\frac{\partial}{\partial x}u$
does not possess a convenient form, thus causes essential difficulty.
While our stochastic equation is driven by a space-time fractional
noises which in some sense alleviates the technical difficulties.

The paper is organized as following. After introducing the stochastic
integration theory for fractional noise and the class of SPDEs in
the next section, we study the well solvability of non-local SPDE,
including existence, uniqueness and regularity of solutions in Section
3. Section 4 is devoted to the existence and uniqueness of the SPDE.
We collect several estimates about Green functions used in the main
text in the Appendix.

Throughout the paper, the generic positive constant $C$ may be different
from line to line.

\section{Preliminaries}

In this part, we first recall a few definitions about the fractional
noise and their stochastic integrals. The technical assumptions which
will be enforced in the present paper are stated clearly, and some
a priori estimates are established.

\subsection{Fractional noise}

A one-dimensional fractional Brownian motion $W^{h}=\{W_{t}^{h},\ t\in\lbrack0,T]\}$
with Hurst parameter $h\in(0,1)$ on $[0,T]$ is a centered Gaussian
process on some probability space $(\Omega,{\mathcal{F}},\Bbb{P})$
with its covariance function given by
\[
\Bbb{E}\left\{ W_{t}^{h}W_{s}^{h}\right\} =\frac{1}{2}\left(t^{2h}+s^{2h}-|t-s|^{2h}\right)\text{.}
\]
The existence of such a Gaussian process and the regularity of its
sample paths are well documented. Other equivalent definitions of
fractional Brownian motion and its analysis may be found in \cite{MV,nualart1}.

Similarly, we may generalize the definition to fractional noises with
two parameters (see also Jiang et al. \cite{JWZ} for further details).

\begin{definition} A one-dimensional double-parameter fractional
Brownian field $B^{H}=\{B^{H}(t,x)$, $(t,x)\in\lbrack0,T]\times D\}$
with Hurst parameter $H=(h_{1},h_{2})$ for $h_{i}\in(0,1)$ and $i\in\{1,2\}$,
where $D=(0,\infty)$, is a centered Gaussian field defined on some
probability space $(\Omega,{\mathcal{F}},\Bbb{P})$ with covariance
\begin{eqnarray}
\Bbb{E}\left\{ B^{H}(t,x)B^{H}(s,y)\right\}  & = & \frac{1}{4}\left(t^{2h_{1}}+s^{2h_{1}}-|t-s|^{2h_{1}}\right)\nonumber \\
 &  & \times\left(x^{2h_{2}}+y^{2h_{2}}-|x-y|^{2h_{2}}\right)\nonumber \\
:= &  & R(t,s;x,y)\label{q1}
\end{eqnarray}
for all $t,s\in\lbrack0,T]$ and $x,y\in D$. \end{definition}

Let $\mathcal{E}$ denote the collection of all step functions defined
on $[0,T]\times D$ and $L_{H}^{2}$ denote the Hilbert space of the
closure of $\mathcal{E}$ under scalar product
\[
\langle I_{[0,t]\times\lbrack0,x]},I_{[0,s]\times\lbrack0,y]}\rangle_{L_{H}^{2}}=R(t,s;x,y)\text{.}
\]
Then the mapping $I_{[0,t]\times\lbrack0,x]}\rightarrow B^{H}(t,x)$
can be extended to an isometry between $L_{H}^{2}$ and the Gaussian
space $\mathcal{H}$ associated with $B^{H}$.

\begin{remark} In this paper we only consider the one-dimensional
double-parameter fractional Brownian field with Hurst parameter $H=(h_{1},h_{2})$,
where $h_{i}\in(\frac{1}{2},1)$, $i=1,2$. \end{remark}

Introduce the square integrable kernel
\[
K_{H}(t,s;x,y)=c_{H}s^{\frac{1}{2}-h_{1}}y^{\frac{1}{2}-h_{2}}\int_{s}^{t}\int_{y}^{x}(u-s)^{h_{1}-\frac{3}{2}}u^{h_{1}-\frac{1}{2}}(z-y)^{h_{2}-\frac{3}{2}}z^{h_{2}-\frac{1}{2}}dzdu
\]
and its derivative
\[
\frac{\partial^{2}}{\partial t\partial x}K_{H}(t,s;x,y)=c_{H}(t-s)^{h_{1}-\frac{3}{2}}\left(\frac{t}{s}\right)^{h_{1}-\frac{1}{2}}(x-y)^{h_{2}-\frac{3}{2}}\left(\frac{x}{y}\right)^{h_{2}-\frac{1}{2}}\text{.}
\]
Define the operator $K_{H}^{\ast}$ from $\mathcal{E}$ to $L^{2}([0,T]\times D)$
by
\[
(K_{H}^{\ast}\phi)(s,y)=\int_{s}^{T}\int_{y}^{\infty}\phi(t,x)\frac{\partial^{2}}{\partial t\partial x}K_{H}(t,s;x,y)dtdx.
\]
It is easy to check that
\[
(K_{H}^{\ast}I_{[0,t]\times\lbrack0,x]})(s,y)=K_{H}(t,s;x,y)I_{[0,t]\times\lbrack0,x]}(s,y),
\]
and
\begin{eqnarray*}
\langle K_{H}^{\ast}I_{[0,t]\times\lbrack0,x]},K_{H}^{\ast}I_{[0,s]\times\lbrack0,y]}\rangle_{L^{2}([0,T]\times D)} & = & R_{H}(t,s;x,y)\\
 & = & \langle I_{[0,t]\times\lbrack0,x]},I_{[0,s]\times\lbrack0,y]}\rangle_{L_{H}^{2}}\text{.}
\end{eqnarray*}
Hence, the operator $K_{H}^{\ast}$ is an isometry between $\mathcal{E}$
and $L^{2}([0,T]\times D)$ which can be extended to $L_{H}^{2}$.
By definition
\[
B(t,x)=B^{H}\left({K_{H}^{\ast}}^{-1}(I_{[0,t]\times\lbrack0,x]})\right),\quad(t,x)\in\lbrack0,T]\times D,
\]
is a Brownian sheet, and in turn the fractional noise has a representation
\[
B^{H}(t,x)=\int_{0}^{t}\int_{0}^{x}K_{H}(t,s;x,y)B(ds,dy)\text{.}
\]

The following embedding property enables us to define the integral
for $\phi\in L_{H}^{2}$ with respect to $B^{H}$.

\begin{proposition} \label{embedding prop} For $h>1/2$, $L^{2}([0,T]\times D)\subset L^{\frac{1}{h}}\left([0,T]\times D\right)\subset L_{H}^{2}$.
\end{proposition}

The integral $\int_{0}^{t}\int_{0}^{x}\phi(s,y)B^{H}(ds,dy)$ is defined
by
\begin{eqnarray}
\int_{0}^{t}\int_{0}^{x}\phi(s,y)B^{H}(ds,dy)=\int_{0}^{t}\int_{0}^{x}(K_{H}^{*}\phi)(s,y)B(ds,dy).\label{fractional integral}
\end{eqnarray}

For $0\leq s<t\leq T$ and $x,y\in D$ define
\[
\Psi_{h}(t,s,x,y):=4h_{1}h_{2}(2h_{1}-1)(2h_{2}-1)|t-s|^{2h_{1}-2}|x-y|^{2h_{2}-2}\text{.}
\]
A routine calculation shows the equivalence of the stochastic integrals
defined in Jiang et al. \cite{JWZ} and those in this section for
functions in $L_{H}^{2}$.

\begin{proposition} For $f,g\in L_{H}^{2}$ we have
\[
\Bbb{E}\int_{0}^{t}{\int_{D}{f(s,x)B^{H}(dx,ds)}}=0
\]
and
\begin{eqnarray*}
 &  & \Bbb{E}\left\{ \int_{0}^{t}{\int_{D}{f(s,x)B^{H}(dx,ds)}}\int_{0}^{t}{\int_{D}{g(s,x)B^{H}(dx,ds)}}\right\} \\
 & = & \int_{[0,t]^{2}}{\int_{D^{2}}{\Psi_{h}(u,v,x,y)f(u,x)g(v,y)dydxdvdu}}\text{.}
\end{eqnarray*}
\end{proposition}

In what follows, $\{{\mathcal{F}}_{t},t\in\lbrack0,T]\}$ denotes
the natural filtration generated by the fractional noise $B^{H}$,
that is, ${\mathcal{F}}_{t}$ is the completion of $\sigma\left\{ B^{H}(s,x),s\leq t,x\in D\right\} $,
which thus satisfies the usual conditions.

\begin{remark} The following embedding lemma (see \cite{memin})
yields directly Proposition \ref{embedding prop}. \end{remark}

\begin{lemma} \label{embedding theorem} If $h\in(\frac{1}{2},1)$
and $f,g\in L^{\frac{1}{h}}([a,b])$, then
\[
\int_{a}^{b}\int_{a}^{b}f(u)g(v)|u-v|^{2h-2}dudv\leq C(h)\Vert f\Vert_{L^{\frac{1}{h}}([a,b])}\Vert g\Vert_{L^{\frac{1}{h}}([a,b])}\text{,}
\]
where $C(h)>0$ is a constant depending only on $h$. \end{lemma}

\subsection{Several technical estimates}

We are concerned with the following SPDE driven by a space-time fractional
Brownian field:
\begin{equation}
\left\{ \begin{array}{ll}
\frac{\partial u}{\partial t}(t,x)=\frac{1}{2}\Delta u(t,x)+g(u(t,x),\frac{\partial}{\partial x}u(t,x))+\dot{B}^{H}\text{,}\\
u(t,0)=0\text{,}\\
u(0,x)=u_{0}(x)\text{,}
\end{array}\right.\label{EQ}
\end{equation}
for $(t,x)\in\lbrack0,T]\times D$, where
\[
\dot{B}^{H}=B^{H}(dt,dx)=\frac{\partial^{2}B^{H}(t,x)}{\partial t\partial x}
\]
is a fractional Brownian field on $(\Omega,\mathcal{F},\Bbb{P})$
with Hurst parameter $H=(h_{1},h_{2})$ for $h_{i}\in(0,1)$ and $i\in\{1,2\}$.

Throughout the remaining part of the paper, the Hurst parameter $H=(h_{1},h_{2})$
satisfies the following hypothesis $H_{h_{1},h_{2}}$:

\begin{hypothesis} \label{hypothesis 1} (1) $h_{i}\in(\frac{1}{2},1)$,
$i=1,2$, and (2) $2h_{1}+h_{2}>2$. \end{hypothesis}

The initial data $u_{0}:D\mapsto\Bbb{R}$ satisfies the following
hypothesis $H_{u_{0}}$:

\begin{hypothesis} \label{hypothesis 2} (1) $||u_{0}||_{\infty}:=\sup_{x}|u_{0}(x)|<\infty$,\\
(2) $||u_{0}^{\prime}||_{\infty}<\infty$,\\
(3) $u_{0}^{\prime}(x)$ is $\kappa$-Hölder continuous in $x$ with
$\kappa\in(0,1)$. \end{hypothesis}

The function $g:\Bbb{R}\times\Bbb{R}\mapsto\Bbb{R}$ satisfies the
following hypothesis $H_{g}$:

\begin{hypothesis} \label{hypothesis 3} There exists a constant
$L>0$ such that $|g(x_{1},y_{1})-g(x_{2},y_{2})|\leq L(|x_{1}-x_{2}|+|y_{1}-y_{2}|)$.
\end{hypothesis}

Let us consider the following non-local SPDE:
\begin{equation}
\left\{ \begin{array}{ll}
\frac{\partial u}{\partial t}(t,x)=\frac{1}{2}\Delta u(t,x)+g(u(t,x),u^{\theta}(t,x))+\dot{B}^{H},\\
u^{\theta}(t,x)=\frac{1}{\theta}\left(u(t,x+\theta)-u(t,x)\right),\\
u(t,0)=0,\\
u(0,x)=u_{0}(x)
\end{array}\right.\label{EQ u}
\end{equation}
where $\theta\in\Bbb{R}$, $\theta\neq0$, $(t,x)\in\lbrack0,T]\times D$
and $(t,x+\theta)\in\lbrack0,T]\times D$.


Suppose $p(t,x,y)$ is the Green function of $\frac{\partial}{\partial t}-\frac{1}{2}\Delta$
in $D$ subject to the Dirichlet boundary condition, that is,
\[
p(t,x,y)=\frac{1}{\sqrt{2\pi t}}\left(e^{-\frac{(x-y)^{2}}{2t}}-e^{-\frac{(x+y)^{2}}{2t}}\right)\text{,}
\]
then we may rewrite Eq. (\ref{EQ u}) as the following
\[
\left\{ \begin{array}{l}
u(t,x)=\int_{D}p(t,x,y)u_{0}(y)dy+\int_{0}^{t}\int_{D}p(t-s,x,y)g(u(s,y),u^{\theta}(s,y))dyds\\
\ \ \ \ \ \ \ \ +\int_{0}^{t}\int_{D}p(t-s,x,y)B^{H}(ds,dy),\\
u^{\theta}(t,x)=\frac{1}{\theta}\left(u(t,x+\theta)-u(t,x)\right).
\end{array}\right.
\]
Note that, from (\ref{fractional integral}),
\[
\int_{0}^{t}\int_{D}p(t-s,x,y)B^{H}(ds,dy)=\int_{0}^{t}\int_{D}(K_{H}^{\ast}p)(t-s,x,y)B(ds,dy)\text{.}
\]

The following lemma provides the key estimate we need in what follows.

\begin{lemma} \label{estimation on integral} Suppose $\psi(t,x)$
is a measurable function, and suppose
\[
\left|H(t,x,y)\right|\leq t^{-\rho}e^{-\frac{C(x-y)^{2}}{t}},
\]
where $\rho<\frac{3}{2}$. Then there exists a constant $C_{T}>0$
such that for $t\in\lbrack0,T]$,
\begin{equation}
\Bbb{E}\left(\int_{0}^{t}\int_{D}H(t-s,x,y)\psi(s,y)dyds\right)^{2}\leq C_{T}\int_{0}^{t}(t-s)^{\frac{1}{2}-\rho}\sup_{y}\Bbb{E}\left(\psi^{2}(s,y)\right)ds\text{.}\label{q2}
\end{equation}
\end{lemma}

\begin{proof} Applying the Cauchy-Schwarz inequality, we get
\begin{eqnarray*}
 &  & \Bbb{E}\left(\int_{0}^{t}\int_{D}H(t-s,x,y)\psi(s,y)dyds\right)^{2}\\
 & \leq & \left(\int_{0}^{t}\int_{D}H(t-s,x,y)dyds\right)\times\Bbb{E}\left(\int_{0}^{t}\int_{D}H(t-s,x,y)\psi^{2}(s,y)dyds\right)\\
 & \leq & C_{T}\int_{0}^{t}(t-s)^{\frac{1}{2}-\rho}\sup_{y}\Bbb{E}\left(\psi^{2}(s,y)\right)ds,
\end{eqnarray*}
which completes the proof. \end{proof}

\begin{remark} \label{remark on p'}The previous estimate is applicable
to
\[
H(t,x,y)=\frac{1}{\theta}\left(p(t,x+\theta,y)-p(t,x,y)\right)\text{,}
\]
so that
\[
\Bbb{E}\left(\int_{0}^{t}\int_{D}H(t-s,x,y)\psi(s,y)dyds\right)^{2}\leq C_{T}\int_{0}^{t}(t-s)^{-\frac{1}{2}}\sup_{y}\Bbb{E}\left(\psi^{2}(s,y)\right)ds.
\]

In fact,
\begin{eqnarray*}
H(t,x,y) & = & \frac{1}{\theta}\left(p(t,x+\theta,y)-p(t,x,y)\right)\\
 & = & \int_{0}^{1}\frac{\partial}{\partial x}p(t,x+a\theta,y)da\text{,}
\end{eqnarray*}
and (see also Lemma \ref{estimation on P})
\[
\frac{\partial}{\partial x}p(t,x,y)\leq Ct^{-1}e^{-\frac{(x-y)^{2}}{4t}}\text{.}
\]
Thus, by the Fubini Theorem, we deduce that
\begin{eqnarray*}
\int_{0}^{t}\int_{D}H(t-s,x,y)dyds & = & \int_{0}^{t}\int_{D}\int_{0}^{1}\frac{\partial}{\partial x}p(t,x+a\theta,y)dadyds\\
 & = & \int_{0}^{1}\left(\int_{0}^{t}\int_{D}\frac{\partial}{\partial x}p(t,x+a\theta,y)dyds\right)da\\
 & \leq & C\int_{0}^{1}t^{\frac{1}{2}}da\leq C_{T}\text{.}
\end{eqnarray*}
Therefore
\begin{eqnarray*}
 &  & \Bbb{E}\left(\int_{0}^{t}\int_{D}H(t-s,x,y)\psi^{2}(s,y)dyds\right)\\
 & = & \Bbb{E}\left(\int_{0}^{t}\int_{D}\int_{0}^{1}\frac{\partial}{\partial x}p(t,x+a\theta,y)\psi^{2}(s,y)dadyds\right)\\
 & = & \Bbb{E}\left(\int_{0}^{1}\left(\int_{0}^{t}\int_{D}\frac{\partial}{\partial x}p(t,x+a\theta,y)\psi^{2}(s,y)dyds\right)da\right)\\
 & \leq & C_{T}\int_{0}^{t}(t-s)^{-\frac{1}{2}}\sup_{y}\Bbb{E}\left(\psi^{2}(s,y)\right)ds\text{.}
\end{eqnarray*}
\hfill{}\\
\end{remark}

We also need an estimate on the second moment of some stochastic integrals.

\begin{lemma} \label{estimation on fractionl integral} Suppose $f(t,x)\in L_{H}^{2}$,
then
\begin{equation}
\Bbb{E}\left[\int_{0}^{t}{\int_{D}\!{f(s,x)B^{H}(dx,ds)}}\right]^{2}\leq C(h_{1},h_{2})\left(\int_{0}^{t}(\Vert f(s,\cdot)\Vert_{L^{\frac{1}{h_{2}}}(D)})^{\frac{1}{h_{1}}}ds\right)^{2h_{1}}\text{.}\label{q4}
\end{equation}
\end{lemma}

\begin{proof} By Proposition \ref{embedding theorem}, we have
\begin{eqnarray*}
 &  & \Bbb{E}\left[\int_{0}^{t}\!\!{\int_{D}\!{f(s,x)B^{H}(dx,ds)}}\right]^{2}\\
 & = & \int_{[0,t]^{2}}\!\!{\int_{D^{2}}\!\!{\Psi_{h}(s_{1},s_{2},y_{1},y_{2})f(s_{1},y_{1})f(s_{2},y_{2})dy_{1}dy_{2}ds_{1}ds_{2}}}\\
 & = & C(h_{1},h_{2})\int_{[0,t]^{2}}\!\!{\int_{D^{2}}\!\!{|s_{1}-s_{2}|^{2h_{1}-2}|y_{1}-y_{2}|^{2h_{2}-2}f(s_{1},y_{1})f(s_{2},y_{2})dy_{1}dy_{2}ds_{1}ds_{2}}}\\
 & \leq & C(h_{1},h_{2})\int_{[0,t]^{2}}|s_{1}-s_{2}|^{2h_{1}-2}\Vert f(s_{1},\cdot)\Vert_{L^{\frac{1}{h_{2}}}(D)}\Vert f(s_{2},\cdot)\Vert_{L^{\frac{1}{h_{2}}}(D)}ds_{1}ds_{2}\\
 & \leq & C(h_{1},h_{2})\left(\int_{0}^{t}(\Vert f(s,\cdot)\Vert_{L^{\frac{1}{h_{2}}}(D)})^{\frac{1}{h_{1}}}ds\right)^{2h_{1}}\text{,}
\end{eqnarray*}
and the proof of this lemma is complete. \end{proof}

\section{Solvability of non-local SPDE}

This section is devoted to the study of non-local SPDE. We study the
uniqueness, existence and regularity of the solution.

\subsection{Uniqueness}

$\Bbb{S}$ denotes the collection of all functions $u:D\times\lbrack0,T]\rightarrow\Bbb{R}$
such that for every $t\in\lbrack0,T]$,
\[
\sup_{x}\Bbb{E}|u(t,x)|^{2}<\infty.
\]
In fact, for the solution set $\Bbb{S}$: $\sup_{x}\Bbb{E}|u(t,x)|^{2}$
is locally integrable in $t$ in order to apply Henry's Gronwall type
inequality later.

\begin{theorem} Suppose $H_{h_{1},h_{2}}$, $H_{u_{0}}$ and $H_{g}$
hold, then there exists at most one solution $(u(t,x),u^{\theta}(t,x))$
of the SPDE (\ref{EQ u}), where $u(t,x)\in\Bbb{S}$ and $u^{\theta}(t,x)\in\Bbb{S}$.
\end{theorem}

\begin{proof} Suppose $(u(t,x),u^{\theta}(t,x))$ and $(\tilde{u}(t,x),\tilde{u}^{\theta}(t,x))$
are two solutions of the equation (2.2), then
\begin{eqnarray*}
 &  & |u(t,x)-\tilde{u}(t,x)|\\
 & = & \left|\int_{0}^{t}\int_{D}p(t-s,x,y)\left[g(u(s,y),u^{\theta}(s,y))-g(\tilde{u}(s,y),\tilde{u}^{\theta}(s,y))\right]dyds\right|\\
 & \leq & C\int_{0}^{t}\int_{D}|p(t-s,x,y)|\left[|u(s,y)-\tilde{u}(s,y)|+|u^{\theta}(s,y)-\tilde{u}^{\theta}(s,y)|\right]dyds\text{,}
\end{eqnarray*}
so that, by Lemma \ref{estimation on integral},
\begin{eqnarray}
\sup_{x}\Bbb{E}|u(t,x)-\tilde{u}(t,x)|^{2} & \leq & C\int_{0}^{t}\sup_{y}\Bbb{E}|u(s,y)-\tilde{u}(s,y)|^{2}ds\nonumber \\
 &  & +C\int_{0}^{t}\sup_{y}\Bbb{E}|u^{\theta}(s,y)-\tilde{u}^{\theta}(s,y)|^{2}ds.\label{q6}
\end{eqnarray}
Since
\begin{eqnarray*}
 &  & u^{\theta}(t,x)-\tilde{u}^{\theta}(t,x)\\
 & = & \int_{0}^{t}\int_{D}\frac{p(t-s,x+\theta,y)-p(t-s,x,y)}{\theta}\\
 &  & \times\left[g(u(s,y),u^{\theta}(s,y))-g(\tilde{u}(s,y),\tilde{u}^{\theta}(s,y))\right]dyds.
\end{eqnarray*}
According to (\ref{q2}), we have
\begin{eqnarray}
 &  & \sup_{x}\Bbb{E}|u^{\theta}(t,x)-\tilde{u}^{\theta}(t,x)|^{2}\nonumber \\
 & \leq & C\int_{0}^{t}(t-s)^{-\frac{1}{2}}\sup_{y}\Bbb{E}|u(s,y)-\tilde{u}(s,y)|^{2}ds\nonumber \\
 &  & +C\int_{0}^{t}(t-s)^{-\frac{1}{2}}\sup_{y}\Bbb{E}|u^{\theta}(s,y)-\tilde{u}^{\theta}(s,y)|^{2}ds.\label{q7}
\end{eqnarray}
Let
\[
\Gamma(t)=\sup_{x}\Bbb{E}|u(t,x)-\tilde{u}(t,x)|^{2}+\sup_{x}\Bbb{E}|u^{\theta}(t,x)-\tilde{u}^{\theta}(t,x)|^{2}\text{.}
\]
Then from (\ref{q6}) and (\ref{q7}), we have
\begin{eqnarray*}
\Gamma(t) & \leq & C\int_{0}^{t}\left(1+(t-s)^{-\frac{1}{2}}\right)\sup_{y}\Bbb{E}|u(s,y)-\tilde{u}(s,y)|^{2}ds\\
 &  & +C\int_{0}^{t}\left(1+(t-s)^{-\frac{1}{2}}\right)\sup_{y}\Bbb{E}|u^{\theta}(s,y)-\tilde{u}^{\theta}(s,y)|^{2}ds\\
 & = & C\int_{0}^{t}\left(1+(t-s)^{-\frac{1}{2}}\right)\Gamma(s)ds,
\end{eqnarray*}
thus it follows from the Gronwall inequality (see e.g. Lemma 1.1 in
\cite{He}) that
\[
\Gamma(t)=0,\ \ as\ \ t\in\lbrack0,T]\text{.}
\]
Hence
\[
\sup_{x}\Bbb{E}|u(t,x)-\tilde{u}(t,x)|^{2}=0,\ \ as\ \ t\in\lbrack0,T],
\]
and therefore
\[
\sup_{x}\Bbb{E}|u^{\theta}(t,x)-\tilde{u}^{\theta}(t,x)|^{2}=0,\ \ as\ \ t\in\lbrack0,T].
\]
That is,
\[
(u(t,x),u^{\theta}(t,x))=(\tilde{u}(t,x),\tilde{u}^{\theta}(t,x)),\ \ (t,x)\in\lbrack0,T]\times D,
\]
in $L^{2}$ sense. The proof thus is completed. \end{proof}

\subsection{Existence}

\begin{theorem} \label{thq1}Suppose that $H_{h_{1},h_{2}}$, $H_{u_{0}}$
and $H_{g}$ hold, then there exists one solution $(u,u^{\theta})$
of the SPDE (\ref{EQ u}), where $u\in\Bbb{S}$ and $u^{\theta}\in\Bbb{S}$.
\end{theorem}

To prove Theorem \ref{thq1}, let us consider the Picard iteration
$\{(u_{n}(t,x),u_{n}^{\theta}(t,x))\}_{n\geq0}$ defined by
\[
\left\{ \begin{array}{ll}
u_{n+1}(t,x)=\int_{D}p(t,x,y)u_{0}(y)dy+\int_{0}^{t}\int_{D}p(t-s,x,y)g(u_{n}(s,y),u_{n}^{\theta}(s,y))dsdy\\
\ \ \ \ \ \ \ \ \ \ \ \ \ \ \ +\int_{0}^{t}\int_{D}p(t-s,x,y)B^{H}(ds,dy)\text{,}\\
u_{n}^{\theta}(t,x)=\frac{1}{\theta}\left(u_{n}(t,x+\theta)-u_{n}(t,x)\right)
\end{array}\right.
\]
where
\[
u_{0}(t,x):=\int_{D}p(t,x,y)u_{0}(y)dy.
\]
If $u_{n}(t,x)\in\Bbb{S}$, then clearly $u_{n}^{\theta}(t,x)\in\Bbb{S}$.
We check that $u_{n+1}\in\Bbb{S}$. Note that
\[
\Bbb{E}\left(\int_{D}p(t,x,y)u_{0}(y)dy\right)^{2}\leq||u_{0}||_{\infty}^{2}\text{,}
\]
so that, by Lemma \ref{estimation on integral}, we have
\begin{eqnarray*}
 &  & \Bbb{E}\left(\int_{0}^{t}\int_{D}p(t-s,x,y)g(u_{n}(s,y),u_{n}^{\theta}(s,y))dsdy\right)^{2}\\
 & \leq & C\left(1+\int_{0}^{t}\sup_{y}\Bbb{E}|u_{n}(s,y)|^{2}ds+\int_{0}^{t}\sup_{y}\Bbb{E}|u_{n}^{\theta}(s,y)|^{2}ds\right)\text{.}
\end{eqnarray*}
While, by Lemma \ref{estimation on fractionl integral}, one gets
\[
\Bbb{E}\left(\int_{0}^{t}\int_{D}p(t-s,x,y)B^{H}(ds,dy)\right)^{2}\leq C(h_{1},h_{2})\left(\int_{0}^{t}(\Vert p(t-s,x,\cdot)\Vert_{L^{\frac{1}{h_{2}}}(D)})^{\frac{1}{h_{1}}}ds\right)^{2h_{1}}
\]
and
\begin{eqnarray*}
\Vert p(t-s,x,\cdot)\Vert_{L^{\frac{1}{h_{2}}}(D)} & = & \left(\int_{D}|p(t-s,x,y)|^{\frac{1}{h_{2}}}dy\right)^{h_{2}}\\
 & \leq & C\left(\int_{D}(t-s)^{-\frac{1}{2h_{2}}}\exp(-\frac{1}{4h_{2}}\frac{|x-y|^{2}}{t-s})dy\right)^{h_{2}}\\
 & \leq & C\left((t-s)^{\frac{1}{2}(1-\frac{1}{h_{2}})}\right)^{h_{2}}\\
 & = & C(t-s)^{\frac{1}{2}(h_{2}-1)}\text{.}
\end{eqnarray*}
Then
\begin{eqnarray}
\Bbb{E}\left(\int_{0}^{t}\int_{D}p(t-s,x,y)B^{H}(ds,dy)\right)^{2} & \leq & C(h_{1},h_{2})\left(\int_{0}^{t}(\Vert p(t-s,x,\cdot)\Vert_{L^{\frac{1}{h_{2}}}(D)})^{\frac{1}{h_{1}}}ds\right)^{2h_{1}}\nonumber \\
 & \leq & C(h_{1},h_{2})t^{2h_{1}+h_{2}-1}\nonumber \\
 & \leq & C(h_{1},h_{2},T)\label{q8}
\end{eqnarray}
and therefore
\begin{eqnarray}
\sup_{x}\Bbb{E}|u_{n+1}(t,x)|^{2} & \leq & C+C\int_{0}^{t}\sup_{y}\Bbb{E}|u_{n}(s,y)|^{2}ds\nonumber \\
 &  & +C\int_{0}^{t}\sup_{y}\Bbb{E}|u_{n}^{\theta}(s,y)|^{2}ds.\label{q21}
\end{eqnarray}

On the other hand,
\begin{eqnarray*}
u_{n+1}^{\theta}(t,x) & = & \int_{D}\frac{1}{\theta}\left(p(t,x+\theta,y)-p(t,x,y)u_{0}(y)\right)dy\\
 &  & +\int_{0}^{t}\int_{D}\frac{1}{\theta}\left(p(t-s,x+\theta,y)-p(t-s,x,y)\right)g(u_{n}(s,y),u_{n}^{\theta}(s,y))dyds\\
 &  & +\int_{0}^{t}\int_{D}\frac{1}{\theta}\left(p(t-s,x+\theta,y)-p(t-s,x,y)\right)B^{H}(ds,dy)\\
 & = & \int_{D}\int_{0}^{1}\frac{\partial}{\partial x}p(t,x+a\theta,y)u_{0}(y)dady\\
 &  & +\int_{0}^{t}\int_{D}\int_{0}^{1}\frac{\partial}{\partial x}p(t-s,x+a\theta,y)g(u_{n}(s,y),u_{n}^{\theta}(s,y))dadyds\\
 &  & +\int_{0}^{t}\int_{D}\int_{0}^{1}\frac{\partial}{\partial x}p(t-s,x+a\theta,y)daB^{H}(ds,dy),
\end{eqnarray*}
and
\begin{eqnarray*}
\Bbb{E}\left(\int_{D}\int_{0}^{1}\frac{\partial}{\partial x}p(t,x+a\theta,y)u_{0}(y)dady\right)^{2} & = & \Bbb{E}\left(\int_{0}^{1}\left(-\int_{D}p(t,x+a\theta,y)u_{0}^{^{\prime}}(y)dy\right)da\right)^{2}\\
 & \leq & ||u_{0}^{^{\prime}}||_{\infty}.
\end{eqnarray*}
By Hypothesis \ref{hypothesis 3} and Remark \ref{remark on p'},
we thus obtain
\begin{eqnarray*}
 &  & \Bbb{E}\left(\int_{0}^{t}\int_{D}\int_{0}^{1}\frac{\partial}{\partial x}p(t-s,x+a\theta,y)g(u_{n}(s,y),u_{n}^{\theta}(s,y))dadyds\right)^{2}\\
 & \leq & C+C\int_{0}^{t}(t-s)^{-\frac{1}{2}}\sup_{y}\Bbb{E}|u_{n}(s,y)|^{2}ds+C\int_{0}^{t}(t-s)^{-\frac{1}{2}}\sup_{y}\Bbb{E}|u_{n}^{\theta}(s,y)|^{2}ds.
\end{eqnarray*}
By a similar argument as those in the proof of (\ref{q8}), we have
\begin{eqnarray}
 &  & \Bbb{E}\left(\int_{0}^{t}\int_{D}\int_{0}^{1}\frac{\partial}{\partial x}p(t-s,x+a\theta,y)daB^{H}(ds,dy)\right)^{2}\nonumber \\
 & \leq & C(h_{1},h_{2})\left(\int_{0}^{t}\left(\left\Vert \int_{0}^{1}\frac{\partial}{\partial x}p(t-s,x+a\theta,\cdot)da\right\Vert _{L^{\frac{1}{h_{2}}}(D)}\right)^{\frac{1}{h_{1}}}ds\right)^{2h_{1}}\nonumber \\
 & \leq & C(h_{1},h_{2})\int_{0}^{t}(t-s)^{\frac{1}{2h_{1}}(h_{2}-2)}ds\nonumber \\
 & \leq & C(h_{1},h_{2},T)\text{,}\label{q9}
\end{eqnarray}
where we have used the assumption that $2h_{1}+h_{2}-2>0$, the Fubini
theorem and Lemma \ref{estimation on P} which yields that
\begin{eqnarray*}
 &  & \left\Vert \int_{0}^{1}\frac{\partial}{\partial x}p(t-s,x+a\theta,\cdot)da\right\Vert _{L^{\frac{1}{h_{2}}}(D)}\\
 & = & \left(\int_{D}\left|\int_{0}^{1}\frac{\partial}{\partial x}p(t-s,x+a\theta,y)da\right|^{\frac{1}{h_{2}}}dy\right)^{h_{2}}\\
 & \leq & C\left(\int_{D}\int_{0}^{1}\left|\frac{\partial}{\partial x}p(t-s,x+a\theta,y)\right|^{\frac{1}{h_{2}}}dady\right)^{h_{2}}\\
 & \leq & C\left(\int_{0}^{1}\left(\int_{D}(t-s)^{-\frac{1}{h_{2}}}\exp(-\frac{1}{4h_{2}}\frac{|x+a\theta-y|^{2}}{t-s})dy\right)da\right)^{h_{2}}\\
 & \leq & C\left((t-s)^{\frac{1}{2h_{2}}(h_{2}-2)}\right)^{h_{2}}\\
 & = & C(t-s)^{\frac{1}{2}(h_{2}-2)}\text{.}
\end{eqnarray*}
Therefore
\begin{eqnarray}
 &  & \sup_{x}\Bbb{E}|u_{n+1}^{\theta}(t,x)|^{2}\nonumber \\
 & \leq & C+C\left(\int_{0}^{t}(t-s)^{-\frac{1}{2}}\sup_{y}\Bbb{E}|u_{n}(s,y)|^{2}ds+\int_{0}^{t}(t-s)^{-\frac{1}{2}}\sup_{y}\Bbb{E}|u_{n}^{\theta}(s,y)|^{2}ds\right).\label{q11}
\end{eqnarray}

Let
\[
\Psi_{n}(t)=\sup_{x}\Bbb{E}|u_{n}(t,x)|^{2}+\sup_{x}\Bbb{E}|u_{n}^{\theta}(t,x)|^{2}
\]
and
\[
\Psi(t)=\limsup_{n}\Psi_{n}(t)\text{.}
\]
Then, by (\ref{q21}) and (\ref{q11}), we get
\begin{eqnarray*}
\Psi_{n+1}(t) & \leq & C+C\int_{0}^{t}\left(1+(t-s)^{-\frac{1}{2}}\right)\sup_{y}\Bbb{E}|u_{n}(s,y)|^{2}ds\\
 &  & +C\int_{0}^{t}\left(1+(t-s)^{-\frac{1}{2}}\right)\sup_{y}\Bbb{E}|u_{n}^{\theta}(s,y)|^{2}ds\\
 & = & C+C\int_{0}^{t}\left(1+(t-s)^{-\frac{1}{2}}\right)\Psi_{n}(s)ds\text{,}
\end{eqnarray*}
and therefore
\[
\Psi(t)\leq C+C\int_{0}^{t}\left(1+(t-s)^{-\frac{1}{2}}\right)\Psi(s)ds\text{.}
\]
By applying the Gronwall inequality, to obtain that $\Psi(t)<\infty$
for $t\in\lbrack0,T]$. It follows that
\[
\sup_{x}\Bbb{E}|u_{n}(t,x)|^{2}<\infty,\ \ as\ \ t\in\lbrack0,T],
\]
and
\[
\sup_{x}\Bbb{E}|u_{n}^{\theta}(t,x)|^{2}<\infty,\ \ as\ \ t\in\lbrack0,T].
\]
Therefore, for any $n$ and $(t,x)\in\lbrack0,T]\times D$,
\[
u_{n}(t,x)\in\Bbb{S},\ \ u_{n}^{\theta}(t,x)\in\Bbb{S}\text{.}
\]

We next prove \ that the sequences $\{(u_{n}(t,x)\}_{n\geq0}$ and
$\{(u_{n}^{\theta}(t,x)\}_{n\geq0}$ are Cauchy sequences in $\Bbb{S}$.
To this end, consider
\[
\left\{ \begin{array}{ll}
u_{n+k+1}(t,x)-u_{n+1}(t,x)\\
\ \ \ =\int_{0}^{t}\int_{D}p(t-s,x,y)\left[g(u_{n+k}(s,y),u_{n+k}^{\theta}(s,y))-g(u_{n}(s,y),u_{n}^{\theta}(s,y))\right]dyds\\
u_{n}^{\theta}(t,x)=\frac{1}{\theta}\left(u_{n}(t,x+\theta)-u_{n}(t,x)\right),
\end{array}\right.
\]
where $k=1,2,3,\cdot\cdot\cdot$. Then, there is a similar estimate
as (\ref{q21}) for the difference
\begin{eqnarray}
 &  & \sup_{x}\Bbb{E}|u_{n+k+1}(t,x)-u_{n+1}(t,x)|^{2}\nonumber \\
 & \leq & C\int_{0}^{t}\sup_{y}\Bbb{E}|u_{n+k}(s,y)-u_{n}(s,y)|^{2}ds\nonumber \\
 &  & +C\int_{0}^{t}\sup_{y}\Bbb{E}|u_{n+k}^{\theta}(s,y)-u_{n}^{\theta}(s,y)|^{2}ds\text{.}\label{q22}
\end{eqnarray}
On the other hand
\begin{eqnarray}
 &  & \sup_{x}\Bbb{E}|u_{n+k+1}^{\theta}(t,x)-u_{n+1}^{\theta}(t,x)|^{2}\nonumber \\
 & \leq & C\int_{0}^{t}(t-s)^{-\frac{1}{2}}\sup_{y}\Bbb{E}|u_{n+k}(s,y)-u_{n}(s,y)|^{2}ds\nonumber \\
 &  & +C\int_{0}^{t}(t-s)^{-\frac{1}{2}}\sup_{y}\Bbb{E}|u_{n+k}^{\theta}(s,y)-u_{n}^{\theta}(s,y)|^{2}ds\text{.}\label{q23}
\end{eqnarray}
So that, by (\ref{q22}) and (\ref{q23}),
\[
\Phi(t)\leq C\int_{0}^{t}\left(1+(t-s)^{-\frac{1}{2}}\right)\Phi(s)ds
\]
where
\[
\Phi(t)=\limsup_{n\rightarrow\infty}\sup_{k}\left(\sup_{x}\Bbb{E}|u_{n+k+1}(t,x)-u_{n+1}(t,x)|^{2}+\sup_{x}\Bbb{E}|u_{n+k+1}^{\theta}(t,x)-u_{n+1}^{\theta}(t,x)|^{2}\right)\text{.}
\]
Once again by Gronwall inequality,
\[
\Phi(t)=0,\ \ t\in\lbrack0,T].
\]
Then
\[
\sup_{x}\Bbb{E}|u_{n+k}(t,x)-u_{n}(t,x)|^{2}\rightarrow0
\]
as $n\rightarrow\infty$, for every $k$ and for all$\ \ t\in\lbrack0,T]$.
Therefore
\[
\sup_{x}\Bbb{E}|u_{n+k}^{\theta}(t,x)-u_{n}^{\theta}(t,x)|^{2}\rightarrow0
\]
as $n\rightarrow\infty$ for$\ t\in\lbrack0,T]$. That is, $\{u_{n}(t,x)\}$
and $\{u_{n}^{\theta}(t,x)\}$ are Cauchy sequences on $\Bbb{S}$.
The limits of these sequences are denoted by $u(t,x)$ and $u^{\theta}(t,x)$
which also belong to $\Bbb{S}$. Therefore the pair $(u(t,x),u^{\theta}(t,x))$
is a solution of the SPDE (\ref{EQ u}).

\subsection{Regularity of the unique solution}

\label{regularity}

Let $(u(t,x),u^{\theta}(t,x))$ be the solution of the stochastic
equation (\ref{EQ u}) under the assumptions as in Theorem \ref{thq1}.
Then $u(t,x)\in\Bbb{S}$ and $u^{\theta}(t,x)\in\Bbb{S}$. We next
discuss the Hölder continuity of $u(t,x)$ and $u^{\theta}(t,x)$.

\begin{theorem} Assume that $H_{h_{1},h_{2}}$, $H_{u_{0}}$ and
$H_{g}$ hold, and that $u(t,x)$ is the solution of the equation
(\ref{EQ u}). Then $u(t,x)$ is $\mu_{1}$-Hölder continuous in $t$
and $\nu_{1}$-Hölder continuous in $x$, where $\mu_{1}\in(0,\frac{1}{2})$
and $\nu_{1}\in(0,1)$. Moreover, $u^{\theta}(t,x)$ is $\mu_{2}$-Hölder
continuous in $t$ and $\nu_{2}$-Hölder continuous in $x$, where
$\mu_{2}\in(0,\min\{\frac{\kappa}{2},\frac{2h_{1}+h_{2}-1}{3}\})$
and $\nu_{2}\in(0,\min\{\kappa,\frac{2h_{1}+h_{2}-1}{2}\})$. \end{theorem}

The remainder of the section is devoted to the proof of the theorem
above.

Without loss of generality, we suppose that $0\leq s\leq t\leq T$
and $0\leq y\leq x$. First observe that
\[
\Bbb{E}(u(t,x)-u(s,y))^{2}\leq2\left(\Bbb{E}(u(t,x)-u(t,y))^{2}+\Bbb{E}(u(t,y)-u(s,y))^{2}\right)\text{.}
\]
It is elementary to see that
\begin{eqnarray*}
 &  & \Bbb{E}(u(t,x)-u(t,y))^{2}\\
 & \leq & C\left(\Bbb{E}\left[\int_{D}(p(t,x,z)-p(t,y,z))u_{0}(z)dz\right]\right)^{2}\\
 &  & +\Bbb{E}\left(\int_{0}^{t}\int_{D}(p(t-r,x,z)-p(t-r,y,z))g(u(r,z),u^{\theta}(r,z))dzdr\right)^{2}\\
 &  & +\Bbb{E}\left(\int_{0}^{t}\int_{D}(p(t-r,x,z)-p(t-r,y,z))B^{H}(dr,dz)\right)^{2}\\
 & = & C(I_{1}+I_{2}+I_{3})\text{.}
\end{eqnarray*}
According to Hypothesis \ref{hypothesis 2}, we have
\begin{eqnarray*}
|u_{0}(x)-u_{0}(y)| & = & |u_{0}^{^{\prime}}(\cdot)||x-y|\\
 & \leq & \Vert u_{0}^{^{\prime}}\Vert_{\infty}|x-y|\text{.}
\end{eqnarray*}
By Lemma \ref{holder on u0},
\[
I_{1}=\Bbb{E}\left(\int_{D}(p(t,x,z)-p(t,y,z))u_{0}(z)dz\right)^{2}\leq C|x-y|^{2}\text{.}
\]
Let us deal with $I_{2}$. Clearly
\begin{eqnarray*}
 &  & \left|\int_{0}^{t}\int_{D}\left|p(t-r,x,z)-p(t-r,y,z)\right|g(u(r,z),u^{\theta}(r,z))dyds\right|\\
 & \leq & C\int_{0}^{t}\int_{D}\left|p(t-r,x,z)-p(t-r,y,z)\right|\left(1+|u(r,z)|+|u^{\theta}(r,z)|\right)|dzdr\\
 & \leq & C\left\{ \int_{0}^{t}\int_{D}\left|p(t-r,x,z)-p(t-r,y,z)\right|dzdr\right.\\
 &  & +\int_{0}^{t}\int_{D}\left|p(t-r,x,z)-p(t-r,y,z)\right||u(r,z)|dzdr\\
 &  & +\left.\int_{0}^{t}\int_{D}\left|p(t-r,x,z)-p(t-r,y,z)\right||u^{\theta}(r,z)|dzdr\right\} \text{.}
\end{eqnarray*}
While, by Lemma \ref{estimation on P} and the Fubini theorem, one
gets
\begin{eqnarray*}
 &  & \Bbb{E}\left(\int_{0}^{t}\int_{D}\left|p(t-r,x,z)-p(t-r,y,z)\right||u(r,z)|dzdr\right)^{2}\\
 & \leq & \int_{0}^{t}\int_{D}\left|p(t-r,x,z)-p(t-r,y,z)\right|dzdr\\
 &  & \times\int_{0}^{t}\int_{D}\left|p(t-r,x,z)-p(t-r,y,z)\right||u(r,z)|^{2}dzdr\\
 & \leq & C\left(\int_{0}^{t}\int_{D}\left|p(t-r,x,z)-p(t-r,y,z)\right|\right)^{2}dzdr\\
 & = & C|x-y|^{2}\left(\int_{0}^{t}\int_{D}\left|\int_{0}^{1}\frac{\partial}{\partial x}p(t-r,y+a(x-y),z)da\right|dzdr\right)^{2}\\
 & \leq & C|x-y|^{2}\left(\int_{0}^{1}\left(\int_{0}^{t}\int_{D}(t-r)^{-1}e^{-\frac{(y+a(x-y)-z)^{2}}{4(t-r)}}dzdr\right)da\right)^{2}\\
 & \leq & C|x-y|^{2}\text{,}
\end{eqnarray*}
and
\begin{eqnarray*}
 &  & \int_{0}^{t}\int_{D}\left|p(t-r,x,z)-p(t-r,y,z)\right||u^{\theta}(r,z)|dzdr\\
 & = & (x-y)\int_{0}^{t}\int_{D}\left|\int_{0}^{1}\frac{\partial}{\partial x}p(t-r,y+a(x-y),z)da\right||u^{\theta}(r,z)|dzdr\text{.}
\end{eqnarray*}
Thus, by using Lemma \ref{estimation on P} and Remark \ref{remark on p'},
we have
\begin{eqnarray*}
 &  & \Bbb{E}\left(\int_{0}^{t}\int_{D}|p(t-r,x,z)-p(t-r,y,z)||u^{\theta}(r,z)|dzdr\right)^{2}\\
 & \leq & C|x-y|^{2}\int_{0}^{t}(t-r)^{-\frac{1}{2}}\sup_{z}\Bbb{E}|u^{\theta}(r,z)|^{2}dr\\
 & \leq & C|x-y|^{2},
\end{eqnarray*}
and therefore,
\[
I_{2}\leq C|x-y|^{2}\text{.}
\]
Next we estimate $I_{3}$. Let $\gamma\in(0,\min\{2h_{1}+h_{2}-1,1\})=(0,1)$.
Then
\begin{eqnarray*}
I_{3} & = & \Bbb{E}\left(\int_{0}^{t}\int_{D}(p(t-r,x,z)-p(t-r,y,z))B^{H}(dr,dz)\right)^{2}\\
 & = & \int_{[0,s]^{2}}\!\!\int_{D^{2}}\Psi_{h}(r,\bar{r},z,\bar{z})|p(t-r,x,z)-p(t-r,y,z)|\\
 &  & \times|p(t-\bar{r},x,\bar{z})-p(t-\bar{r},y,\bar{z})|dzd\bar{z}drd\bar{r}\\
 & = & \left\Vert p(t-\cdot,x,\cdot)-p(t-\cdot,y,\cdot)\right\Vert _{L_{H}^{2}}^{2}\\
 & = & \left\Vert |p(t-\cdot,x,\cdot)-p(t-\cdot,y,\cdot)|^{\gamma}|p(t-\cdot,x,\cdot)-p(t-\cdot,y,\cdot)|^{1-\gamma}\right\Vert _{{L_{H}^{2}}}^{2}\\
 & \leq & C(\gamma)\left(\left\Vert |p(t-\cdot,x,\cdot)-p(t-\cdot,y,\cdot)|^{\gamma}|p(t-\cdot,x,\cdot)|^{1-\gamma}\right\Vert _{{L_{H}^{2}}}^{2}\right.\\
 &  & \left.+\left\Vert |p(t-\cdot,x,\cdot)-p(t-\cdot,y,\cdot)|^{\gamma}|p(t-\cdot,y,\cdot)|^{1-\gamma}\right\Vert _{{L_{H}^{2}}}^{2}\right)\\
 & := & C(\gamma)(I_{31}+I_{32})\text{.}
\end{eqnarray*}
On other hand, according to Lemma \ref{estimation on P} and the Fubini
theorem, one can get
\begin{eqnarray*}
I_{31} & \leq & \left\Vert \left|\int_{0}^{1}\frac{\partial}{\partial x}p(t-\cdot,y+a(x-y),\cdot)da\right|^{\gamma}|x-y|^{\gamma}|p(t-\cdot,x,\cdot)|^{1-\gamma}\right\Vert _{L_{H}^{2}}^{2}\\
 & = & |x-y|^{2\gamma}\int_{[0,T]^{2}}\!\!\int_{D^{2}}\left|\int_{0}^{1}\frac{\partial}{\partial x}p(t-r,y+a(x-y),z)da\right|^{\gamma}|p(t-r,x-z)|^{1-\gamma}\\
 &  & \times\Psi_{h}(r,\bar{r},z,\bar{z})\left|\int_{0}^{1}\frac{\partial}{\partial x}p(t-\bar{r},y+a(x-y),\bar{z})da\right|^{\gamma}|p(t-\bar{r},x-\bar{z})|^{1-\gamma}dzd\bar{z}drd\bar{r}\\
 & \leq & C(h_{1},h_{2},\gamma)|x-y|^{2\gamma}\\
 &  & \times\left(\int_{0}^{T}\left(\int_{D}\left(\left|\int_{0}^{1}\frac{\partial}{\partial x}p(t-r,y+a(x-y),z)da\right|^{\gamma}|p(t-r,x-z)|^{1-\gamma}\right)^{\frac{1}{h_{2}}}dz\right)^{\frac{h_{2}}{h_{1}}}dr\right)^{2h_{1}}\\
 & \leq & C(h_{1},h_{2},\gamma)|x-y|^{2\gamma}\int_{0}^{1}\left(\int_{0}^{T}(t-r)^{\frac{-\gamma-\frac{1}{2}(1-\gamma)}{h_{1}}}(t-r)^{\frac{h_{2}}{2h_{1}}}dr\right)^{2h_{1}}da\\
 & \leq & C(h_{1},h_{2},\gamma)|x-y|^{2\gamma}\left(\int_{0}^{T}(t-r)^{\frac{h_{2}-1-\gamma}{2h_{1}}}dr\right)^{2h_{1}}\\
 & \leq & C(h_{1},h_{2},\gamma)|x-y|^{2\gamma}\text{.}
\end{eqnarray*}
Similarly, we have
\[
I_{32}\leq C(T,h_{1},h_{2})|x-y|^{2\gamma}\text{.}
\]
Therefore we deduce that
\begin{eqnarray}
I_{3}\leq C(T,h_{1},h_{2})|x-y|^{2\gamma}\text{,}\label{q312}
\end{eqnarray}
hence
\[
\Bbb{E}(u(t,x)-u(t,y))^{2}\leq C|x-y|^{2\nu_{1}}
\]
for $\nu_{1}\in(0,\min\{2h_{1}+h_{2}-1,1\})=(0,1)$.

By a similar argument as above, we can get
\[
\Bbb{E}(u(t,y)-u(s,y))^{2}\leq C|t-s|^{2\mu_{1}}
\]
for $\mu_{1}\in(0,\frac{1}{2}\min\{2h_{1}+h_{2}-1,1\})=(0,\frac{1}{2})$.

That is, $u(t,x)$ is $\mu_{1}$-Hölder continuous in $t$ and $\nu_{1}$-Hölder
continuous in $x$, where $\mu_{1}\in(0,\frac{1}{2})$ and $\nu_{1}\in(0,1)$.

On the other hand, for any $\theta>0$, we recall
\[
u^{\theta}(t,x)=\frac{1}{\theta}(u(t,x+\theta)-u(t,x)),
\]
and $u^{\theta}(t,x)\in\Bbb{S}$, and
\begin{eqnarray}
u^{\theta}(t,x) & = & \int_{D}\frac{1}{\theta}\left(p(t,x+\theta,z)-p(t,x,z)\right)u_{0}(z)dz\nonumber \\
 &  & +\int_{0}^{t}\int_{D}\frac{1}{\theta}\left(p(t-s,x+\theta,z)-p(t-s,x,z)\right)g(u(s,z),u^{\theta}(s,z))dzds\nonumber \\
 &  & +\int_{0}^{t}\int_{D}\frac{1}{\theta}\left(p(t-s,x+\theta,z)-p(t-s,x,z)\right)B^{H}(ds,dz)\nonumber \\
 & = & \int_{D}\int_{0}^{1}\frac{\partial}{\partial x}p(t,x+a\theta,z)u_{0}(z)dadz\nonumber \\
 &  & +\int_{0}^{t}\int_{D}\int_{0}^{1}\frac{\partial}{\partial x}p(t-s,x+a\theta,z)g(u(s,z),u^{\theta}(s,z))dadzds\nonumber \\
 &  & +\int_{0}^{t}\int_{D}\int_{0}^{1}\frac{\partial}{\partial x}p(t-s,x+a\theta,z)daB^{H}(ds,dz).\label{q12}
\end{eqnarray}
Thus
\begin{eqnarray*}
 &  & u^{\theta}(t,x)-u^{\theta}(t,y)\\
 & = & \int_{D}\left(\ \int_{0}^{1}\frac{\partial}{\partial x}p(t,x+a\theta,z)da-\int_{0}^{1}\frac{\partial}{\partial x}p(t,y+a\theta,z)da\right)u_{0}(z)dz\\
 &  & +\int_{0}^{t}\int_{D}\left(\int_{0}^{1}\frac{\partial}{\partial x}p(t-r,x+a\theta,z)da-\int_{0}^{1}\frac{\partial}{\partial x}p(t-r,y+a\theta,z)da\right)g(u(r,z),u^{\theta}(r,z))dzdr\\
 &  & +\int_{0}^{t}\int_{D}\left(\int_{0}^{1}\frac{\partial}{\partial x}p(t-r,x+a\theta,z)da-\int_{0}^{1}\frac{\partial}{\partial x}p(t-r,y+a\theta,z)da\right)B^{H}(dr,dz).
\end{eqnarray*}
By Hypothesis \ref{hypothesis 2} and Lemma \ref{holder on u0}, one
gets
\begin{eqnarray}
 &  & \Bbb{E}\left(\int_{D}\left(\int_{0}^{1}\frac{\partial}{\partial x}p(t,x+a\theta,z)da-\int_{0}^{1}\frac{\partial}{\partial x}p(t,y+a\theta,z)da\right)u_{0}(z)dz\right)^{2}\label{holder on u' of x}\\
 & = & \Bbb{E}\left(\int_{0}^{1}\left(-\int_{D}\left(p(t,x+a\theta,z)-p(t,y+a\theta,z)\right)u_{0}^{^{\prime}}(z)dz\right)da\right)^{2}\nonumber \\
 & \leq & C|x-y|^{2\kappa}.\nonumber
\end{eqnarray}

Moreover
\begin{eqnarray*}
 &  & \Bbb{E}\left(\int_{0}^{t}\int_{D}\left(\int_{0}^{1}\frac{\partial}{\partial x}p(t-r,x+a\theta,z)da-\int_{0}^{1}\frac{\partial}{\partial x}p(t-r,y+a\theta,z)da\right)g(u(r,z),u^{\theta}(r,z))dzdr\right)^{2}\\
 & \leq & \int_{0}^{t}\int_{D}\left|\int_{0}^{1}\frac{\partial}{\partial x}p(t-r,x+a\theta,z)da-\int_{0}^{1}\frac{\partial}{\partial x}p(t-r,y+a\theta,z)da\right|dzdr\\
 &  & \times\int_{0}^{t}\int_{D}\left|\int_{0}^{1}\frac{\partial}{\partial x}p(t-r,x+a\theta,z)da-\int_{0}^{1}\frac{\partial}{\partial x}p(t-r,y+a\theta,z)da\right|\Bbb{E}(g(u(r,z),u^{\theta}(r,z)))^{2}dzdr\\
 & \leq & C\left(\int_{0}^{t}\int_{D}\left|\int_{0}^{1}\frac{\partial}{\partial x}p(t-r,x+a\theta,z)da-\int_{0}^{1}\frac{\partial}{\partial x}p(t-r,y+a\theta,z)da\right|dzdr\right)^{2}\\
 & = & C\left\{ \int_{0}^{t}\int_{D}\left|\int_{0}^{1}\frac{\partial}{\partial x}p(t-r,x+a\theta,z)da-\int_{0}^{1}\frac{\partial}{\partial x}p(t-r,y+a\theta,z)da\right|^{\varrho}\right.\\
 &  & \times\left.\left|\int_{0}^{1}\frac{\partial}{\partial x}p(t-r,x+a\theta,z)da-\int_{0}^{1}\frac{\partial}{\partial x}p(t-r,y+a\theta,z)da\right|^{1-\varrho}dzdr\right\} ^{2}\\
 & \leq & C|x-y|^{2\varrho}\left\{ \int_{0}^{t}\int_{D}\left|\int_{0}^{1}\int_{0}^{1}\frac{\partial^{2}}{\partial x^{2}}p(t-r,y+b(x-y)+a\theta,z)dadb\right|^{\varrho}\right.\\
 &  & \times\left.\left|\int_{0}^{1}\frac{\partial}{\partial x}p(t-r,x+a\theta,z)da-\int_{0}^{1}\frac{\partial}{\partial x}p(t-r,y+a\theta,z)da\right|^{1-\varrho}dzdr\right\} ^{2}\\
 & \leq & C|x-y|^{2\varrho}\left\{ \left[\int_{0}^{t}\int_{D}\left|\int_{0}^{1}\int_{0}^{1}\frac{\partial^{2}}{\partial x^{2}}p(t-r,y+b(x-y)+a\theta,z)dadb\right|^{\varrho}\right.\right.\\
 &  & \times\left.\left|\int_{0}^{1}\frac{\partial}{\partial x}p(t-r,x+a\theta,z)da\right|^{1-\varrho}dzdr\right]^{2}\\
 &  & +\left.\left[\int_{0}^{t}\int_{D}\left|\int_{0}^{1}\int_{0}^{1}\frac{\partial^{2}}{\partial x^{2}}p(t-r,y+b(x-y)+a\theta,z)dadb\right|^{\varrho}\left|\int_{0}^{1}\frac{\partial}{\partial x}p(t-r,y+a\theta,z)da\right|^{1-\varrho}dzdr\right]^{2}\right\} ,
\end{eqnarray*}
where $\varrho\in(0,1)$.

While, from Lemma \ref{estimation on P} and Fubini theorem,
\begin{eqnarray*}
 &  & \int_{0}^{t}\int_{D}\left|\int_{0}^{1}\int_{0}^{1}\frac{\partial^{2}}{\partial x^{2}}p(t-r,y+b(x-y)+a\theta,z)dadb\right|^{\varrho}\left|\int_{0}^{1}\frac{\partial}{\partial x}p(t-r,x+a\theta,z)da\right|^{1-\varrho}dzdr\\
 & \leq & C\int_{0}^{t}\int_{D}\int_{0}^{1}\int_{0}^{1}\int_{0}^{1}\left|\frac{\partial^{2}}{\partial x^{2}}p(t-r,y+b(x-y)+a\theta,z)\right|^{\varrho}\left|\frac{\partial}{\partial x}p(t-r,x+c\theta,z)\right|^{1-\varrho}dadbdcdzdr\\
 & \leq & C\int_{0}^{1}\int_{0}^{1}\int_{0}^{1}\left(\int_{0}^{t}\int_{D}\left|(t-r)^{-\frac{3}{2}}e^{-\frac{(y+b(x-y)+a\theta-z)^{2}}{4t}}\right|^{\varrho}\left|(t-r)^{-1}e^{-\frac{(x+c\theta-z)^{2}}{4(t-r)}}\right|^{1-\varrho}dzdr\right)dadbdc\\
 & \leq & C\int_{0}^{1}\int_{0}^{1}\int_{0}^{1}\left(\int_{0}^{t}\int_{D}(t-r)^{-1-\frac{1}{2}\varrho}e^{-\frac{(y+b(x-y)+a\theta-z)^{2}}{4(t-r)}}drdz\right)dadbdc\\
 & \leq & C\int_{0}^{t}(t-r)^{-\frac{1}{2}-\frac{1}{2}\varrho}dr\\
 & < & \infty,
\end{eqnarray*}
with $-\frac{1}{2}-\frac{1}{2}\varrho>-1$ as $\varrho<1$.

Similarly, by using the argument as in the proof of (\ref{q312}),
we have
\begin{eqnarray*}
 &  & \Bbb{E}\left(\int_{0}^{t}\int_{D}\left(\int_{0}^{1}\frac{\partial}{\partial x}p(t-r,x+a\theta,z)da-\int_{0}^{1}\frac{\partial}{\partial x}p(t-r,y+a\theta,z)da\right)B^{H}(dr,dz)\right)^{2}\\
 & = & \left\Vert \int_{0}^{1}\frac{\partial}{\partial x}p(t-\cdot,x+a\theta,\cdot)da-\int_{0}^{1}\frac{\partial}{\partial x}p(t-\cdot,y+a\theta,\cdot)da\right\Vert _{L_{H}^{2}}^{2}\\
 & = & \left\Vert \left|\int_{0}^{1}\frac{\partial}{\partial x}p(t-\cdot,x+a\theta,\cdot)da-\int_{0}^{1}\frac{\partial}{\partial x}p(t-\cdot,y+a\theta,\cdot)da\right|^{\gamma^{^{\prime}}}\right.\\
 &  & \times\left.\left|\int_{0}^{1}\frac{\partial}{\partial x}p(t-\cdot,x+a\theta,\cdot)da-\int_{0}^{1}\frac{\partial}{\partial x}p(t-\cdot,y+a\theta,\cdot)da\right|^{1-{\gamma^{^{\prime}}}}\right\Vert _{L_{H}^{2}}^{2}\\
 & \leq & C(h_{1},h_{2},\gamma)|x-y|^{2\gamma^{^{\prime}}}
\end{eqnarray*}
where $\gamma^{^{\prime}}\in(0,\min\{\frac{2h_{1}+h_{2}-1}{2},1\})=(0,\frac{2h_{1}+h_{2}-1}{2})$.

Putting together the estimates above, we deduce that
\begin{equation}
\Bbb{E}(u^{\theta}(t,x)-u^{\theta}(t,y))^{2}\leq C|x-y|^{2\nu_{2}}\label{q24}
\end{equation}
where $\nu_{2}\in(0,\min\{\kappa,\frac{2h_{1}+h_{2}-1}{2},1\})=(0,\min\{\kappa,\frac{2h_{1}+h_{2}-1}{2}\})$.

Next let us deal with the difference
\begin{eqnarray*}
 &  & u^{\theta}(t,x)-u^{\theta}(s,x)\\
 & = & \int_{D}\left(\int_{0}^{1}\frac{\partial}{\partial x}p(t,x+a\theta,z)da-\int_{0}^{1}\frac{\partial}{\partial x}p(s,x+a\theta,z)da\right)u_{0}(z)dz\\
 &  & +\int_{s}^{t}\int_{D}\int_{0}^{1}\frac{\partial}{\partial x}p(t-r,x+a\theta,z)g(u(r,z),u^{\theta}(r,z))dadzdr\\
 &  & +\int_{0}^{s}\int_{D}\left(\int_{0}^{1}\frac{\partial}{\partial x}p(t-r,x+a\theta,z)da-\int_{0}^{1}\frac{\partial}{\partial x}p(s-r,x+a\theta,z)da\right)g(u(r,z),u^{\theta}(r,z))dzdr\\
 &  & +\int_{0}^{t}\int_{D}\left(\int_{0}^{1}\frac{\partial}{\partial x}p(t-r,x+a\theta,z)da-\int_{0}^{1}\frac{\partial}{\partial x}p(s-r,x+a\theta,z)da\right)B^{H}(dr,dz).
\end{eqnarray*}
Using the same approach to (\ref{holder on u' of x}), together with
Lemma \ref{holder on u0} and the Fubini theorem, we obtain
\[
\Bbb{E}\left[\left(\int_{0}^{1}\frac{\partial}{\partial x}p(t,x+a\theta,z)da-\int_{0}^{1}\frac{\partial}{\partial x}p(s,x+a\theta,z)da\right)u_{0}(z)dz\right]^{2}\leq C|t-s|^{\kappa}\text{.}
\]
On the other hand, by Lemma \ref{estimation on P},
\[
\left|\frac{\partial^{2}}{\partial x\partial t}p(t,x,y)\right|\leq Ct^{-2}e^{-\frac{(x-y)^{2}}{4t}}.
\]
So that
\[
\Bbb{E}\left(\int_{s}^{t}\int_{D}\int_{0}^{1}\frac{\partial}{\partial x}p(t-r,x+a\theta,z)g(u(r,z),u^{\theta}(r,z))dadzdr\right)^{2}\leq C|t-s|^{2\sigma}\text{,}
\]
\begin{eqnarray*}
 &  & \Bbb{E}\left(\int_{0}^{s}\int_{D}\int_{0}^{1}\left(\frac{\partial}{\partial x}p(t-r,x+a\theta,z)-\frac{\partial}{\partial x}p(s-r,x+a\theta,z)\right)g(u(r,z),u^{\theta}(r,z))dadzdr\right)^{2}\\
 & \leq & C|t-s|^{2\sigma}
\end{eqnarray*}
and
\begin{eqnarray*}
 &  & \Bbb{E}\left(\int_{0}^{1}\left(\frac{\partial}{\partial x}p(t-r,x+a\theta,z)-\frac{\partial}{\partial x}p(s-r,x+a\theta,z)\right)daB^{H}(dr,dz)\right)^{2}\\
 & \leq & C(T,h_{1},h_{2})|t-s|^{2\iota}\text{,}
\end{eqnarray*}
where $\sigma\in(0,\frac{1}{2})$ and $\iota\in\frac{2h_{1}+h_{2}-1}{3}$.
Therefore
\[
\Bbb{E}(u^{\theta}(t,x)-u^{\theta}(s,x))^{2}\leq C|t-s|^{2\mu_{2}},
\]
where
\[
\mu_{2}\in\left\{ 0,\min\left(\frac{\kappa}{2},\frac{1}{2},\frac{2h_{1}+h_{2}-1}{3}\right)\right\} =\left\{ 0,\min\left(\frac{\kappa}{2},\frac{2h_{1}+h_{2}-1}{3}\right)\right\} \text{.}
\]

Thus we finish the proof of the theorem.

\section{Well Solvability of SPDE}

\setcounter{equation}{0} \global\long\def\theequation{4.\arabic{equation}}

In this part, we study SPDE (\ref{EQ}). Let $v(t,x)=\frac{\partial}{\partial x}u(t,x)$.
Then the SPDE (\ref{EQ}) has the following equivalence expression:
\begin{equation}
\left\{ \begin{array}{ll}
\frac{\partial u}{\partial t}(t,x)=\frac{1}{2}\Delta u(t,x)+g(u(t,x),v(t,x))+\dot{B}^{H},\\
u(t,0)=0,\\
u(0,x)=u_{0}(x),
\end{array}\right.\label{EQ uv}
\end{equation}
which in turn means that the pair $(u(t,x),v(t,x))$ satisfies the
coupled stochastic integral system:
\begin{equation}
\left\{ \begin{array}{ll}
u(t,x)=\int_{D}p(t,x,y)u_{0}(y)dy+\int_{0}^{t}\int_{D}p(t-s,x,y)g(u(s,y),v(s,y))dyds\\
\ \ \ \ \ \ \ \ +\int_{0}^{t}\int_{D}p(t-s,x,y)B^{H}(ds,dy),\\
v(t,x)=\int_{D}\frac{\partial}{\partial x}p(t,x,y)u_{0}(y)dy+\int_{0}^{t}\int_{D}\frac{\partial}{\partial x}p(t-s,x,y)g(u(s,y),v(s,y))dyds\\
\ \ \ \ \ \ \ \ +\int_{0}^{t}\int_{D}\frac{\partial}{\partial x}p(t-s,x,y)B^{H}(ds,dy).
\end{array}\right.\label{EQ u and v}
\end{equation}

\subsection{Existence}


\begin{theorem} Suppose the assumptions $H_{h_{1},h_{2}}$, $H_{u_{0}}$
and $H_{g}$ hold, then SPDE (\ref{EQ}) possesses a solution in $\Bbb{S}$.
\end{theorem}

In fact, we only need to show that SPDE ({\ref{EQ uv}}) possesses
a solution $(\bar{u}(t,x),\bar{v}(t,x))$, where $\bar{u}(t,x)\in\Bbb{S}$
and $\bar{v}(t,x)\in\Bbb{S}$.

As we have demonstrated, if$(u(t,x),u^{\theta}(t,x))$ is the solution
of the equation (\ref{EQ u}), then $u(t,x)\in\Bbb{S}$ and $u^{\theta}(t,x)\in\Bbb{S}$.

Let us consider
\[
\Bbb{E}(u^{\theta_{1}}(t,x)-u^{\theta_{2}}(t,x))^{2}
\]
as $\theta_{1}\rightarrow0$ and $\theta_{2}\rightarrow0.$

According to (\ref{q11}), for $\tau_{1},\tau_{2}\in(0,1)$, we may
write
\begin{eqnarray*}
 &  & u^{\theta_{1}}(t,x)-u^{\theta_{2}}(t,x)\\
 & = & \int_{D}\left(\int_{0}^{1}\frac{\partial}{\partial x}p(t,x+a\theta_{1},y)da-\int_{0}^{1}\frac{\partial}{\partial x}p(t,x+a\theta_{2},y)da\right)u_{0}(y)dy\\
 &  & +\int_{0}^{t}\int_{D}\left\{ \int_{0}^{1}\frac{\partial}{\partial x}p(t-s,x+a\theta_{1},y)g(u(s,y),u^{\theta_{1}}(s,y))da\right.\\
 &  & \ \ \ \ \ \ \ \ \ \ \ \ \ \ \ \ -\left.\int_{0}^{1}\frac{\partial}{\partial x}p(t-s,x+a\theta_{2},y)g(u(s,y),u^{\theta_{2}}(s,y))da\right\} dyds\\
 &  & +\int_{0}^{t}\int_{D}\left\{ \int_{0}^{1}\frac{\partial}{\partial x}p(t-s,x+a\theta_{1},y)da-\int_{0}^{1}\frac{\partial}{\partial x}p(t-s,x+a\theta_{2},y)da\right\} B^{H}(ds,dy)\\
 & := & B_{1}+B_{2}+B_{3}.
\end{eqnarray*}
Let us derive estimates for $B_{1}$, $B_{2}$ and $B_{3}$. Note
that
\begin{eqnarray*}
 &  & \int_{0}^{1}\frac{\partial}{\partial x}p(t-s,x+a\theta_{1},y)g(u(s,y),u^{\theta_{1}}(s,y))da\\
 &  & -\int_{0}^{1}\frac{\partial}{\partial x}p(t-s,x+a\theta_{2},y)g(u(s,y),u^{\theta_{2}}(s,y))da\\
 & = & \int_{0}^{1}\frac{\partial}{\partial x}p(t-s,x+a\theta_{1},y)da\left(g(u(s,y),u^{\theta_{1}}(s,y))-g(u(s,y),u^{\theta_{2}}(s,y))\right)\\
 &  & +\left(\int_{0}^{1}\frac{\partial}{\partial x}p(t-s,x+a\theta_{1},y)da-\int_{0}^{1}\frac{\partial}{\partial x}p(t-s,x+a\theta_{2},y)da\right)g(u(s,y),u^{\theta_{2}}(s,y))
\end{eqnarray*}
and it follows that
\begin{eqnarray*}
B_{2} & = & \int_{0}^{t}\int_{D}\int_{0}^{1}\frac{\partial}{\partial x}p(t-s,x+a\theta_{1},y)\left(g(u(s,y),u^{\theta_{1}}(s,y))-g(u(s,y),u^{\theta_{2}}(s,y))\right)dadyds\\
 &  & +\int_{0}^{t}\int_{D}\left(\int_{0}^{1}\frac{\partial}{\partial x}p(t-s,x+a\theta_{1},y)da-\int_{0}^{1}\frac{\partial}{\partial x}p(t-s,x+a\theta_{2},y)da\right)g(u(s,y),u^{\theta_{2}}(s,y))dyds\\
 & := & B_{21}+B_{22}.
\end{eqnarray*}
By the same argument as in the proof of (\ref{q24}), we may conclude
that
\[
B_{1}\leq C|\theta_{1}-\theta_{2}|^{2\kappa},
\]
\[
B_{22}\leq C|\theta_{1}-\theta_{2}|^{2\varrho},
\]
and
\[
B_{3}\leq C|\theta_{1}-\theta_{2}|^{2\gamma^{^{\prime}}}.
\]
Again by Lemma \ref{estimation on P}, one gets
\[
B_{21}\leq C\int_{0}^{t}(t-s)^{-\frac{1}{2}}\sup_{y}\Bbb{E}\left(u^{\theta_{1}}(s,y)-u^{\theta_{2}}(s,y)\right)^{2}ds.
\]
Combining the estimates above, we have
\begin{eqnarray}
 &  & \Bbb{E}\left(u^{\theta_{1}}(t,x)-u^{\theta_{2}}(t,x)\right)^{2}\nonumber \\
 & \leq & C\left|\theta_{1}-\theta_{2}\right|^{\varpi}+C\int_{0}^{t}(t-s)^{-\frac{1}{2}}\sup_{y}\Bbb{E}\left(u^{\theta_{1}}(s,y)-u^{\theta_{2}}(s,y)\right)^{2}ds,\label{h1-h2}
\end{eqnarray}
where $\varpi=\min\{2\kappa,2\varrho,2\gamma^{^{\prime}}\}$.

Sending $\theta_{1}\rightarrow0$ and $\theta_{2}\rightarrow0$ and
using Gronwall inequality, we get
\[
\sup_{x}\Bbb{E}\left(u^{\theta_{1}}(t,x)-u^{\theta_{2}}(t,x)\right)^{2}\rightarrow0,\ \ t\in\lbrack0,T].
\]
Therefore, $\{u^{\theta}(t,x)\}_{\theta}$ is a Cauchy sequence on
$\Bbb{S}$. The limit of these sequences exists (also belong to $\Bbb{S}$),
which is $\bar{v}(t,x)$.

Finally letting $\theta\rightarrow0$, we denote the limit of $(u(t,x),u^{\theta}(t,x))$
by $(\bar{u}(t,x),\bar{v}(t,x))$, which is the solution of the SPDE
(\ref{EQ uv}). The proof of this theorem is thus complete.

\subsection{Uniqueness}

\begin{theorem} Suppose the assumptions $H_{h_{1},h_{2}}$, $H_{u_{0}}$
and $H_{g}$ hold, then SPDE (\ref{EQ}) has a unique solution in
$\Bbb{S}$. \end{theorem}

\begin{proof} Suppose $(\bar{u}_{1}(t,x),\bar{v}_{1}(t,x))$ and
$(\bar{u}_{2}(t,x),\bar{v}_{2}(t,x))$ are two solutions on the equation
(\ref{EQ u and v}). Then
\begin{eqnarray*}
\bar{u}_{1}(t,x) & = & \int_{D}p(t,x,y)u_{0}(y)dy+\int_{0}^{t}\int_{D}p(t-s,x,y)g(\bar{u}_{1}(s,y),\bar{v}_{1}(s,y))dyds\\
 &  & +\int_{0}^{t}\int_{D}p(t-s,x,y)B^{H}(ds,dy),
\end{eqnarray*}
and
\begin{eqnarray*}
\bar{u}_{2}(t,x) & = & \int_{D}p(t,x,y)u_{0}(y)dy+\int_{0}^{t}\int_{D}p(t-s,x,y)g(\bar{u}_{2}(s,y),\bar{v}_{2}(s,y))dyds\\
 &  & +\int_{0}^{t}\int_{D}p(t-s,x,y)B^{H}(ds,dy).
\end{eqnarray*}
Then
\begin{eqnarray*}
 &  & \bar{u}_{1}(t,x)-\bar{u}_{2}(t,x)\\
 & = & \int_{0}^{t}\int_{D}p(t-s,x,y)\left(g(\bar{u}_{1}(s,y),\bar{v}_{1}(s,y))-g(\bar{u}_{2}(s,y),\bar{v}_{2}(s,y))\right)dyds.
\end{eqnarray*}
Once again by using Lemma \ref{estimation on P} and Lemma \ref{estimation on integral},
\begin{eqnarray}
 &  & \sup_{x}\Bbb{E}(\bar{u}_{1}(t,x)-\bar{u}_{2}(t,x))^{2}\nonumber \\
 & \leq & \int_{0}^{t}\sup_{y}\Bbb{E}(\bar{u}_{1}(s,y)-\bar{u}_{2}(s,y))^{2}ds+\int_{0}^{t}\sup_{y}\Bbb{E}(\bar{v}_{1}(s,y)-\bar{v}_{2}(s,y))^{2}ds.\label{u1-u2}
\end{eqnarray}
On the other hand,
\begin{eqnarray*}
\bar{v}_{1}(t,x) & = & \int_{D}\frac{\partial}{\partial x}p(t,x,y)u_{0}(y)dy+\int_{0}^{t}\int_{D}\frac{\partial}{\partial x}p(t-s,x,y)g(\bar{u}_{1}(s,y),\bar{v}_{1}(s,y))dyds\\
 &  & +\int_{0}^{t}\int_{D}\frac{\partial}{\partial x}p(t-s,x,y)B^{H}(ds,dy),
\end{eqnarray*}
and
\begin{eqnarray*}
\bar{v}_{2}(t,x) & = & \int_{D}\frac{\partial}{\partial x}p(t,x,y)u_{0}(y)dy+\int_{0}^{t}\int_{D}\frac{\partial}{\partial x}p(t-s,x,y)g(\bar{u}_{2}(s,y),\bar{v}_{2}(s,y))dyds\\
 &  & +\int_{0}^{t}\int_{D}\frac{\partial}{\partial x}p(t-s,x,y)B^{H}(ds,dy).
\end{eqnarray*}

Note that
\[
\bar{v}_{i}(t,x)=\lim_{\theta\rightarrow0}\bar{u}_{i}^{\theta}(t,x)\in\Bbb{S},
\]
with $i=1,2.$ Similarly, by a similar argument for (\ref{h1-h2}),
we get
\begin{eqnarray}
 &  & \sup_{x}\Bbb{E}\left(\bar{v}_{1}(t,x)-\bar{v}_{2}(t,x)\right)^{2}\nonumber \\
 & \leq & C\int_{0}^{t}(t-s)^{-\frac{1}{2}}\sup_{y}\Bbb{E}\left(\bar{v}_{1}(s,y)-\bar{v}_{2}(s,y)\right)^{2}ds\nonumber \\
 &  & +C\int_{0}^{t}(t-s)^{-\frac{1}{2}}\sup_{y}\Bbb{E}\left(\bar{u}_{1}(s,y)-\bar{u}_{2}(s,y)\right)^{2}ds.\label{lim uh1-lim uh2}
\end{eqnarray}
Let
\[
\Lambda(t)=\sup_{x}\Bbb{E}(\bar{u}_{1}(t,x)-\bar{u}_{2}(t,x))^{2}+\sup_{x}\Bbb{E}\left(\bar{v}_{1}(t,x)-\bar{v}_{2}(t,x)\right)^{2}.
\]
Jointing with (\ref{u1-u2}) and (\ref{lim uh1-lim uh2}), we get
\[
\Lambda(t)\leq C\int_{0}^{t}\left(1+(t-s)^{-\frac{1}{2}}\right)\Lambda(s)ds.
\]
So
\[
\Lambda(t)=0,\ \ \ as\ \ \ t\in\lbrack0,T].
\]
Then
\[
(\bar{u}_{1}(t,x),\bar{v}_{1}(t,x))=(\bar{u}_{2}(t,x),\bar{v}_{2}(t,x)),\ \ \ as\ \ \ (t,x)\in\lbrack0,T]\times D,
\]
in $L^{2}$ sense. Then we get the result of this theorem. \\
\end{proof}

\subsection{Regularity}

Let $(u(t,x),v(t,x))$ be the solution of the equation (\ref{EQ u and v}).
Then $u(t,x)\in\Bbb{S}$ and $v(t,x)\in\Bbb{S}$, and $u(t,x)$ is
the solution of the equation (\ref{EQ}). By similar arguments as
in the proof of the Hölder continuity of the solution pair $(u(t,x),u^{\theta}(t,x))$
to the equation (\ref{EQ u}) in Section \ref{regularity}, one can
show the Hölder continuity of $u(t,x)$ and $v(t,x)$ which we state
as the following theorem, its proof is omitted.

\begin{theorem} \label{thq3}Assume that $H_{h_{1},h_{2}}$, $H_{u_{0}}$
and $H_{g}$ hold. Let $u(t,x)$ be the solution of the equation (\ref{EQ}).
Then $u(t,x)$ is $\mu_{1}$-Hölder continuous in $t$ and $\nu_{1}$-Hölder
continuous in $x$, where $\mu_{1}\in(0,\frac{1}{2})$ and $\nu_{1}\in(0,1)$.
Moreover, $v(t,x)=\frac{\partial}{\partial x}u(t,x)$ is $\mu_{2}$-Hölder
continuous in $t$ and $\nu_{2}$-Hölder continuous in $x$, where
$\mu_{2}\in(0,\min\{\frac{\kappa}{2},\frac{2h_{1}+h_{2}-1}{3}\})$
and $\nu_{2}\in(0,\min\{\kappa,\frac{2h_{1}+h_{2}-1}{2}\})$. \end{theorem}

\section{Appendix}

\setcounter{equation}{0} \global\long\def\theequation{A.\arabic{equation}}

In this section, we review, for the convenience of the reader, a few
elementary estimates about the Green function which are used in the
paper. Recall that $p(t,x,y)$ is the fundamental solution of the
heat operator $\frac{\partial}{\partial t}-\frac{1}{2}\Delta$ on
$[0,\infty)$ subject to the Dirichlet boundary condition, given by
the following explicit formula
\[
p(t,x,y)=\frac{1}{\sqrt{2\pi t}}\left(e^{-\frac{(x-y)^{2}}{2t}}-e^{-\frac{(x+y)^{2}}{2t}}\right).
\]

\begin{lemma} \label{estimation on P} For $(t,x,y)\in[0,T]\times D\times D$,
we have
\begin{eqnarray*}
\left|p(t,x,y)\right|\leq Ct^{-\frac{1}{2}}e^{-\frac{(x-y)^{2}}{4t}},
\end{eqnarray*}
\begin{eqnarray*}
\left|\frac{\partial}{\partial x}p(t,x,y)\right|\leq Ct^{-1}e^{-\frac{(x-y)^{2}}{4t}},
\end{eqnarray*}
\begin{eqnarray*}
\left|\frac{\partial}{\partial t}p(t,x,y)\right|\leq Ct^{-\frac{3}{2}}e^{-\frac{(x-y)^{2}}{4t}},
\end{eqnarray*}
\begin{eqnarray*}
\left|\frac{\partial^{2}}{\partial x^{2}}p(t,x,y)\right|\leq Ct^{-\frac{3}{2}}e^{-\frac{(x-y)^{2}}{4t}},
\end{eqnarray*}
and
\begin{eqnarray*}
\left|\frac{\partial^{2}}{\partial x\partial t}p(t,x,y)\right|\leq Ct^{-2}e^{-\frac{(x-y)^{2}}{4t}}.
\end{eqnarray*}
\end{lemma}

Let us for example prove the second one, and the proofs for others
are similar. Since
\[
\frac{\partial}{\partial x}p(t,x,y)=\frac{1}{\sqrt{2\pi t}}\left(e^{-\frac{(x-y)^{2}}{2t}}\times\frac{y-x}{t}+e^{-\frac{(x+y)^{2}}{2t}}\times\frac{y+x}{t}\right).
\]
Let $y-x=\xi\sqrt{t}$. Then
\begin{eqnarray*}
\left|\frac{1}{\sqrt{2\pi t}}e^{-\frac{(x-y)^{2}}{2t}}\frac{y-x}{t}\right| & = & \left|\frac{1}{\sqrt{2\pi t}}e^{-\frac{(x-y)^{2}}{4t}}\times e^{-\frac{\xi^{2}}{4}}\frac{\xi}{\sqrt{t}}\right|\\
 & \leq & Ct^{-1}e^{-\frac{(x-y)^{2}}{4t}}.
\end{eqnarray*}
Similarly
\begin{eqnarray*}
\frac{1}{\sqrt{2\pi t}}e^{-\frac{(x+y)^{2}}{2t}}\frac{y+x}{t} & \leq & Ct^{-1}e^{-\frac{(x+y)^{2}}{4t}}\\
 & \leq & Ct^{-1}e^{-\frac{(x-y)^{2}}{4t}}
\end{eqnarray*}
and the proof is complete. \hfill{}$\Box$\\

Similarly, as in Bally et al. \cite{BMS}, we have the following result.

\begin{lemma} \label{holder on u0} Let $u_{0}$ be a $\omega$-Hölder
continuous real function with $0<\omega\leq1$. Then
\[
\left|\int_{D}p(t,x,z)u_{0}(z)dz-\int_{D}p(s,y,z)u_{0}(z)dz\right|\leq C\left(|t-s|^{\frac{\omega}{2}}+|x-y|^{\omega}\right)
\]
for any $s,t\in\lbrack0,T]$ and $x,y\in D=\lbrack0,\infty)$. \end{lemma}

In fact, by the semigroup property of $p(t,x,y)$, we have
\begin{eqnarray*}
 &  & \int_{D}p(t,x,z)u_{0}(z)dz-\int_{D}p(s,x,z)u_{0}(z)dz\\
 & = & \int_{D}\int_{D}p(s,x,y)p(t-s,y,z)u_{0}(z)dydz-\int_{D}p(s,x,y)u_{0}(y)dy\\
 & = & \int_{D}p(s,x,y)\left(\int_{D}p(t-s,y,z)(u_{0}(z)-u_{0}(y))dz\right)dy,
\end{eqnarray*}
so that
\begin{eqnarray*}
 &  & \left|\int_{D}p(t,x,z)u_{0}(z)dz-\int_{D}p(s,x,z)u_{0}(z)dz\right|\\
 & \leq & C\int_{D}p(s,x,y)\left(\int_{D}p(t-s,y,z)|z-y|^{\omega}dz\right)dy\\
 & \leq & C\int_{D}p(s,x,y)|t-s|^{\frac{\omega}{2}}dy\\
 & = & C|t-s|^{\frac{\omega}{2}}.
\end{eqnarray*}
For simplicity, set
\[
\varphi(t,x)=\frac{1}{\sqrt{2\pi t}}e^{-\frac{x^{2}}{2t}}.
\]
Then
\[
p(t,x,y)=\varphi(t,x-y)-\varphi(t,x+y).
\]
If $y>x>0$ and $\lambda=y-x$, then
\begin{eqnarray*}
 &  & \left|\int_{D}p(t,x,z)u_{0}(z)dz-\int_{D}p(t,y,z)u_{0}(z)dz\right|\\
 & = & \left|\int_{D}(\varphi(t,z-x)-\varphi(t,z-y))u_{0}(z)dz\right.\\
 &  & -\left.\int_{D}(\varphi(t,z+x)-\varphi(t,z+y))u_{0}(z)dz\right|\\
 & = & \left|\int_{D}\varphi(t,z-x)(u_{0}(z)-u_{0}(z+\lambda))dz+\int_{-\lambda}^{0}\varphi(t,z-x)u_{0}(z+\lambda)dz\right.\\
 &  & -\left.\int_{\lambda}^{+\infty}\varphi(t,z+x)(u_{0}(z)-u_{0}(z-\lambda))dz-\int_{0}^{\lambda}\varphi(t,z+x)u_{0}(z)dz\right|\\
 & \leq & \left|\int_{D}\varphi(t,z-x)(u_{0}(z)-u_{0}(z+\lambda))dz\right|\\
 &  & +\left|\int_{\lambda}^{+\infty}\varphi(t,z+x)(u_{0}(z)-u_{0}(z-\lambda))dz\right|\\
 &  & +\left|\int_{0}^{\lambda}\varphi(t,z+x)u_{0}(z)dz-\int_{-\lambda}^{0}\varphi(t,z-x)u_{0}(z+\lambda)dz\right|\\
 & \leq & C\lambda^{\omega}\int_{D}p(t,x,z)dz+\left|\int_{\lambda}^{0}\varphi(t,z+x)(u_{0}(z)-u_{0}(\lambda-z))dz\right|\\
 & \leq & C\lambda^{\omega}+C\int_{0}^{\lambda}\varphi(t,z+x)|2z-\lambda|^{\omega}dz\\
 & \leq & C\lambda^{\omega}=C|x-y|^{\omega}.
\end{eqnarray*}
This completes the proof of the lemma.

\begin{acknowledgement}{} The research of Y. Jiang was supported
by the LPMC at Nankai University and the NSF of China (no. 11101223 and 11271203). Z. Qian would like
to thank for the support of an ERC\ research grant for the research
carried out in this paper. \end{acknowledgement}

\end{document}